\documentclass[12pt]{article}
\usepackage{mathrsfs}
\usepackage{tipa}
\usepackage{amssymb}
\usepackage{amsfonts}
\usepackage{cite}
\usepackage{graphicx}
\usepackage{latexsym,amsmath,amssymb,amsfonts,amsthm}

\newtheorem{theorem}{Theorem}[section]
\newtheorem{lemma}[theorem]{Lemma}

\newtheorem{proposition}[theorem]{Proposition}

\begin{document}
\setcounter{page}{1}
\title{The dual foliation of polar actions on nonnegatively curved manifolds}
\author{Yi Shi}
\date{}
\protect
\maketitle ~~~\\[-5mm]

\protect\footnotetext{\!\!\!\!\!\!\!\!\!\!\!\!\! {\bf 2020 Mathematics Subject Classification:}
53C12, 53C20.
}
\maketitle ~~~\\[-5mm]
{\footnotesize School of Mathematics and Statistics, Guizhou University of Finance and Economics, Guiyang 550025, China, e-mail: shiyi\underline{\hbox to 0.2cm{}}math@163.com}\\[1mm]

{\bf Abstract:} \noindent We prove that any dual leaf $L^{\#}$ of a simply connected, complete nonnegatively curved polar manifold $M$ is totally geodesic and closed in $M$, and $L^{\#}$ is itself a complete nonnegatively curved polar manifold. Furthermore, the dual foliation on $M$ induces a Riemannian submersion with totally geodesic fibers from $M$ to a homogeneous space.
\markright{\sl\hfill \hfill}\\

\section{Introduction}
\renewcommand{\thesection}{\arabic{section}}
\renewcommand{\theequation}{\thesection.\arabic{equation}}
\setcounter{equation}{0}

A proper and isometric action of a Lie group $G$ on a complete Riemannian manifold $M$ is called \emph{polar} if there exists a connected complete isometrically immersed submanifold, called a \emph{section}, meeting all orbits and always orthogonally. We call $M$ a \emph{polar manifold} or \emph{polar $G$-manifold}. Polar actions have nice properties and have been studied by many people, see for instance
\cite{Be,BDT,Bi,CO,Da,DDK,EH,FGT,Go04,Go22,GK,Goz,GZ,HPTT,Ko02,Ko09,Ko17,KL,Me,Mu,PaT,PoT}. For further details we refer the readers to the survey article by Gorodski \cite{Go22}.

In this paper, our interest is to study the dual foliation of polar actions on complete Riemannian manifolds. The concept of dual foliation and its foundations were pioneered by Wilking in \cite{wi}. Recall that a \emph{singular Riemannian foliation} $\mathcal{F}$ on a Riemannian manifold $M$ is a decomposition of $M$ into smooth injectively immersed submanfolds $L(x)$, called \emph{leaves}, such that it is a singular foliation and any geodesic starting orthogonally to a leaf remains orthogonal to all leaves it intersects. Such a geodesic is called a \emph{horizontal geodesic}. A piecewise smooth curve is called \emph{horizontal} if it meets the leaves of $\mathcal {F}$ perpendicularly. Suppose a Lie group $G$ acts isometrically on manifold $M$, then the orbit decomposition of the $G$-action gives a singular Riemannian foliation $\mathcal{F}$ on $M$. We call a curve horizontal or $G$-horizontal if it meets the $G$-orbits perpendicularly. We say that $M$ is $G$-\emph{horizontal connected} if any two points in $M$ can be connected by a $G$-horizontal curve. For further details of singular Riemannian foliations, we refer the readers to \cite{ABT,GR,GW,KR,LMR,Ly,LiR,LyR,MR,mol,Ra,wi}.

In \cite{wi}, Wilking associates to a given singular Riemannian foliation $\mathcal{F}$ the so-called \emph{dual foliation} $\mathcal{F}^{\#}$. The \emph{dual leaf} $L^{\#}_{ p}$ through a point $p \in M$ is defined as all points $q \in M$ such that there is a piecewise smooth, horizontal curve from $p$ to $q$. Wilking proved several fundamental properties for the dual foliation $\mathcal{F}^{\#}$ of a singular Riemannian foliation $\mathcal{F}$ on complete nonnegatively curved manifold $M$. Here we just cite some results in \cite{wi} that related to this paper: (1) When $M$ has positive curvature, the horizontal connectivity holds on $M$, i.e., $\mathcal{F}^{\#}$ has only one leaf; (2) If the dual leaves are complete, then $\mathcal{F}^{\#}$ is a singular Riemannian foliation as well; (3) If $\mathcal{F}$ is given by the orbit decomposition of an isometric group action, then the dual leaves are complete. Dual foliation has been used in different situations in literature, see for instance \cite{AGW,GuW,LW,sy1,SS,Sp,wi}.

There are many applications of dual foliation in the study of nonnegatively curved manifolds. Wilking \cite{wi} used it to show that the Sharafutdinov projection is smooth. Very recently, Lytchak and Wilking \cite{LW} used dual foliation to prove that a Riemannian submersion between smooth, compact, nonnegatively curved Riemannian manifolds has to be smooth, resolving a conjecture by Berestovskii-Guijarro \cite{BG}.

The theory of dual foliation also plays an important role in the study of nonnegatively curved manifolds with polar actions. Fang, Grove and Thorbergsson \cite{FGT} proved that a polar action on a simply connected, compact, positively curved manifold of cohomogeneity at least 2 is equivariantly diffeomorphic to a polar action on a compact rank one symmetric space. Lytchak \cite{Ly} proved that a polar foliation of codimension at least three in an irreducible compact symmetric space is hyperpolar, unless the symmetric space has rank one. Both of the two articles used the the theory of dual foliation and Tits building theory (\cite {Ro,Ti}). As pointed out in \cite{FGT} , the classification of polar actions on nonnegatively curved manifolds is significantly more involved. One reason is that the horizontal connectivity holds on positively curved polar manifolds, but not holds on nonnegatively curved polar manifolds in general. In this paper, we overcome this difficulty by proving the following reduction theorem:

\begin{theorem}\label{dl} Suppose a compact Lie group $G$ acts polarly on a simply connected, complete nonnegatively curved manifold $M$.
For any $p\in M$, denote by $H=\{g\in G\ |\ gL_{p}^{\#}=L_{p}^{\#}\}$ the stabilizer group of $L_{p}^{\#}$. Then:

\noindent(1) $L_{p}^{\#}$ is totally geodesic and closed in $M$.

\noindent(2) $H$ is a compact Lie group acting polarly on $L_{p}^{\#}$, and $L_{p}^{\#}$ is itself a $H$-horizontal connected, complete nonnegatively curved polar $H$-manifold.

\noindent(3) The orbit spaces of $(M, G)$ and $(L_{p}^{\#}, H)$ are isometric, i.e., $M/G=L_{p}^{\#}/H$.

\end{theorem}

Part of Theorem \ref{dl} holds without assumption about curvature, see Theorem \ref{generate}.
Combine Theorem \ref{dl} with Proposition 9.1 in \cite{wi} we get the following:
\begin{theorem}\label{df}
Suppose a compact Lie group $G$ acts polarly on a simply connected, complete nonnegatively curved manifold $M$. Let $\mathcal{F}$ denote the singular Riemannian foliation induced by the orbit decomposition of $G$. Then one of the following occurs:

\noindent(1) The dual foliation $\mathcal{F}^{\#}$ has only one leaf, i.e., $M$ is $G$-horizontal connected.

\noindent(2) There is a closed subgroup $H\subsetneq G$, an invariant metric on $G/H$, and a $G$-equivariant Riemannian submersion $\sigma: M\rightarrow G/H$ with totally geodesic fibers. Furthermore, any fiber $B$ of $\sigma$ is a leaf of $\mathcal{F}^{\#}$, and $B$ is isometric to a $H$-horizontal connected, complete nonnegatively curved polar $H$-manifold.
\end{theorem}

Theorem \ref{dl} and \ref{df} show that we can always reduce a simply connected, complete nonnegatively curved polar $G$-manifold $M$ to a $H$-horizontal connected, complete nonnegatively curved polar $H$-manifold $L_{p}^{\#}$.

\noindent$\mathbf{Funding.}$ The author is supported by National Natural Science Foundation of China (No.12261013) and innovative exploration and academic seedling project of Guizhou University of Finance and Economics (No.2024XSXMA09).

\noindent$\mathbf{Acknowledgement.}$ I would like to thank the referee for the careful review and critical comments that improved the exposition.

\section{Preliminaries}
\renewcommand{\thesection}{\arabic{section}}
\renewcommand{\theequation}{\thesection.\arabic{equation}}
\setcounter{equation}{0}
\setcounter{theorem}{0}

In this section we recall known properties of polar actions. Suppose a Lie group $G$ acts polarly on a complete Riemannian manifold $M$. Let $M_{o}$ denote the union of \emph{regular points} in $M$, i.e., $M_{o}$ is the union of all principal orbits in $M$. We denote by $G^{o}$ the identity component of $G$. Fix a section $\Sigma$ of the polar action $(M, G)$, a point $p\in \Sigma\cap M_{o}$. Set $N_{G}(\Sigma):=\{g\in G|\ g\Sigma=\Sigma\}$ and $Z_{G}(\Sigma):=\{g\in G|\ gx=x, \forall x\in\Sigma\}$. We call $W_{G}(\Sigma):= N_{G}(\Sigma)/Z_{G}(\Sigma)$ the \emph{polar group} of $(M, G)$ associated to $\Sigma$.

Most of results in the following proposition are proved in \cite{PaT}, see also for \cite{Go22,GZ}:

\begin{proposition}\label{p1} Let $(M, G)$ be a polar action with a section $\Sigma$, and $W_{G}(\Sigma)$ the polar group of $(M, G)$ associated to $\Sigma$. Then:

\noindent(1) Any section of $(M, G)$ can be written as $g\Sigma$ for some $g\in G$, and any section is totally geodesic in $M$.

\noindent(2) There exists a unique section through a regular point $p\in M_{o}$, and it is given by $\exp_{p}(T^{\perp}_{p}(G(p)))$.

\noindent(3) The polar group $W_{G}(\Sigma)$ is a discrete subgroup of $N(G_{p})/G_{p}$ and acts isometrically and properly discontinuously on $\Sigma$, where $p\in \Sigma\cap M_{o}$ and $N(G_{p})$ is the normalizer of $G_{p}$ in $G$. Furthermore, the orbit spaces $\Sigma/W_{G}(\Sigma)$ and $M/G$ are isometric and thus $M/G$ is an orbifold.

\noindent(4) For any $p\in \Sigma$, the slice representation $(T^{\perp}_{p}(G(p)), G_{p})$ at $p$ is also polar with a section $T_{p}\Sigma$. Also, the isotropy subgroup $G_{p}$ acts transitively on the set of sections of $(M, G)$ through $p$.
\end{proposition}

The following Proposition will be used in our proof:

\begin{proposition}\label{p2}(Proposition 2.4 in \cite{HPTT}) Let $M$ be a Riemannian $G$-manifold. A submanifold $\Sigma$ of $M$ is a section for the action of $G$ if and only if it is a section for the action of $G_{o}$. In particular, the action of $G$ on $M$ is polar if and only if the action of $G_{o}$ on $M$ is polar.
\end{proposition}

For a section $\Sigma$ of $(M, G)$ with polar group $W_{G}(\Sigma)$, there is a non-trivial normal subgroup $W\subset W_{G}(\Sigma)$ generated
by \emph{reflections}. We refer to any group generated by reflections as a \emph{reflection group}. An isometric action $r: \Sigma\rightarrow \Sigma$ is called a reflection if $r$ has order 2, and at least one component of the fixed-point set has codimension 1. The codimension-1 components of the fixed-point set are referred to as the \emph{mirrors} of $r$. A connected component $c$ of the complement of all mirrors is called an \emph{open chamber} of $\Sigma$. We denote the closure of an open chamber by $C =\bar{c}$ and refer to it simply as a \emph{chamber}. The boundary $\partial C=C\backslash c$ of a chamber $C$ is the union of its \emph{chamber faces}, where a chamber face is an intersection $C\cap \Lambda$ with a mirror $\Lambda$ such that $C\cap \Lambda$ is the closure of a non-empty open subset of $\Lambda$. We will frequently simply write ``faces" when we mean ``chamber faces".

The following result is proved more generally for singular polar foliations in [\cite{AT}, Theorem 1.5] and [\cite{Al}, Theorem 1.1]:

\begin{proposition}\label{AT}(\cite{Al, AT}) Any non-trivial polar action on a simply connected manifold has no exceptional orbits, and its reflection group $W$ is the whole polar group $W_{G}(\Sigma)$.
\end{proposition}

By Proposition \ref{AT} above, the following result is a special case of Proposition 2.3 in \cite{GZ}:

\begin{proposition}\label{p3}(Proposition 2.3 in \cite{GZ}) Let $G$ be a polar action on a simply connected manifold $M$ with a section $\Sigma$, and $W_{G}(\Sigma)$ the polar group of $(M, G)$ associated to $\Sigma$. Then $W_{G}(\Sigma)$ acts freely and simply transitively on the set of chambers. The chamber $C$ is a convex set isometric to $\Sigma/W_{G}(\Sigma)= M/G$. Furthermore, the isotropy group along the interior of a chamber face $F_{i}$ is constant equal to the isotropy group of $F_{i}$.
\end{proposition}

The following properties about chamber system are showed in page 9 in \cite{FGT},  here we just cite these properties.  By Proposition \ref{AT} and Proposition \ref{p3}, we get that $G$ acts transitively on the set of all chambers in all sections of $M$, i.e., $M = \cup_{g\in G}gC$. The chamber faces $F_{i},\ i = 1, \ldots, k$, of $C$ correspond to a set of generators $r_{i}$ for $W=W_{G}(\Sigma)$. This way all faces of chambers $gC, g\in G$ of $M$ get labeled consistently, so that $G$ is label preserving. Now define two chambers $g_{1}C$ and $g_{2}C$ to be $i$-adjacent if they have a common $i$ face $g_{1}F_{i} = g_{2}F_{i}$. This relation among the set of chambers in $M$, respectively all chambers in a fixed section $\Sigma$ make both of these sets into a chamber system $\Phi(M, G)$, respectively $\Phi(\Sigma, W_{G}(\Sigma))$ according to the following definition (see, e.g.,\cite{Ro,Ti}):

An \emph{(abstract) chamber system over} $I = \{1, \ldots, k\}$ is a set $\Phi$ together with a partition of $\Phi$ for every $i\in I$. Elements in the same part of the $i$-partition, are said to be \emph{$i$-adjacent}. The elements of $\Phi$ are called \emph{chambers}. A \emph{gallery} in $\Phi$ is a sequence $\Gamma= (C_{0}, \ldots, C_{m})$ such that $C_{j}$ is $i_{j}$-adjacent to $C_{j+1}$ for every $0 \leq j \leq m-1$.

The gallery is a key tool in the proof of Theorem \ref{generate}.

\section{Proof of Theorems}
\renewcommand{\thesection}{\arabic{section}}
\renewcommand{\theequation}{\thesection.\arabic{equation}}
\setcounter{equation}{0}
\setcounter{theorem}{0}

\begin{lemma}\label{S=H}
Let $G$ be a compact Lie group, acting isometrically on a complete Riemannian manifold $M$. For any $p\in M$, set $S:=\{g\in G\ |\ gp\in L_{p}^{\#}\}$ and $H:=\{g\in G\ |\ gL_{p}^{\#}=L_{p}^{\#}\}$. Then $S=H$ and $L_{p}^{\#}\cap G(x)=H(x)$ for any $x\in L_{p}^{\#}$. Also, $gL_{p}^{\#}=L_{gp}^{\#}$ for any $g\in G$.
\end{lemma}

\begin{proof} By Lemma 3.1 in \cite{sy1} $S$ is a subgroup of $G$ and the action of $S$ on $M$ preserves $L_{p}^{\#}$. Now it is easy to see that $S=H$. For any $p\in M$ and $g\in G$, we claim that $gL_{p}^{\#}=L_{gp}^{\#}$. For any $x\in L_{p}^{\#}$, one can connect $p$ and $x$ by a horizontal curve $\gamma$. Then $gx\in L_{gp}^{\#}$, since one can connect $gp$ and $gx$ by the horizontal curve $g\gamma$. Thus $gL_{p}^{\#}\subset L_{gp}^{\#}$. We also have $g^{-1}L_{gp}^{\#}\subset L_{g^{-1}gp}^{\#}=L_{p}^{\#}$. Thus $gL_{p}^{\#}=L_{gp}^{\#}$ as claimed. For any $x\in L_{p}^{\#}$ and $h\in H$, $hx\in hL_{p}^{\#}=L_{hp}^{\#}=L_{p}^{\#}$ since $hp\in L_{p}^{\#}$. Thus $H(x)\subset L_{p}^{\#}\cap G(x)$. For any $h\in G$ with $hx\in L_{p}^{\#}\cap G(x)$, $h\in S=H$ since $hp\in hL_{x}^{\#}=L_{hx}^{\#}=L_{p}^{\#}$. Hence $L_{p}^{\#}\cap G(x)\subset H(x)$. Thus $L_{p}^{\#}\cap G(x)=H(x)$ for any $x\in L_{p}^{\#}$.
\end{proof}

\begin{theorem}\label{generate}
Suppose a compact Lie group $G$ acts polarly on a simply connected, complete Riemannian manifold $M$. Let $p\in M_{o}$ be a regular point, and $C_{0}$ the unique chamber through $p$. Denote by $H=\{g\in G\ |\ gL_{p}^{\#}=L_{p}^{\#}\}$ the stabilizer group of $L_{p}^{\#}$. Then:

\noindent(1) The identity component $H^{o}$ of $H$ is generated by the identity components of face isotropy groups of $C_{0}$. Also, $H$ and $H^{o}$ are compact Lie groups.

\noindent(2) The action of $H$ and $H^{o}$ on $L_{p}^{\#}$ are polar, and $H(x)=H^{o}(x)$ for all $x\in L_{p}^{\#}$. Also, $L_{p}^{\#}$ is $H$-horizontal connected.

\noindent(3) $L_{p}^{\#}$ is closed in $M$.

\noindent(4) The orbit spaces of $(M, G)$ and $(L_{p}^{\#}, H)$ are both isometric to $C_{0}$, i.e., $M/G=C_{0}=L_{p}^{\#}/H$.
\end{theorem}

\begin{proof}\noindent(1) Here we use the same method as Theorem 3.2 in \cite{FGT} to prove (1). Let $K\subset G$ be the group generated by the identity components of face isotropy groups of $C_{0}$. First, we claim that $K\subset H$. Let $H_{i}$ be the isotropy group of the face $F_{i}$ of $C_{0}$, and $H_{i}^{o}$ the identity component of $H_{i}$. By Proposition \ref{p3}, for any point $q$ in the interior of $F_{i}$, we have $G_{q}=G_{F_{i}}=H_{i}$. Let $\gamma\subset C_{0}$ be a horizontal geodesic from $q$ to $p$. Then for any $g\in G_{q}=H_{i}$, $g\gamma$ is a horizontal geodesic from $q$ to $gp$. Thus $gp\in L_{p}^{\#}$, and $g\in S=\{g\in G\ |\ gp\in L_{p}^{\#}\}$. By Lemma 3.1 we have $H_{i}\subset S=H$. This proves that $K\subset H$. Hence $K\subset H^{o}$.

Next, we claim that $H^{o}\subset K$. For any $y\in L_{p}^{\#}\cap G(p)$, we have $y=gp$ for some $g\in S=H$. Let $\alpha$ be a horizontal curve from $p$ to $y$. Since $G$ acts transitively on the set of all chambers in all sections of $M$, i.e. $M = \cup_{g\in G}gC_{0}$, there is a gallery $\Gamma=(C_{0},\ldots, C_{k})$ of type $i_{1}i_{2}\ldots i_{k}$ that covers $\alpha$, i.e. $\alpha\subset \cup_{0\leq i\leq k}C_{i}$, where $p\in C_{0}$ and $y\in C_{k}$. Since any regular point is in a unique chamber and $y=gp\in M_{o}$, we get that $C_{k}=gC_{0}$. Let $\bar{F}_{i_{j}}$ be the common face of $C_{j-1}$ and $C_{j}\ (1\leq j\leq k)$, and $x_{j}$ a interior point of $\bar{F}_{i_{j}}$. Then $G_{x_{j}}=G_{\bar{F}_{i_{j}}}=\bar{H}_{i_{j}}$ is the isotropy group of the face $\bar{F}_{i_{j}}$. By Proposition \ref{p2} $G^{o}_{x_{j}}=\bar{H}^{o}_{i_{j}}$ acts polarly on $T^{\perp}_{x_{j}}(G({x_{j}}))$. Thus $\bar{H}^{o}_{i_{j}}$ acts transitively on the set of sections of $(G, M)$ through $x_{j}$. Since any chamber is on a unique section, we get that $\bar{H}^{o}_{i_{j}}$ acts transitively on the set of chambers through $x_{j}$. Thus $C_{j}=g_{i_{j}}C_{j-1}$ for some $g_{i_{j}}\in \bar{H}^{o}_{i_{j}}$. Set $g_{j}:=g_{i_{j}}g_{i_{j-1}}\cdots g_{i_{1}}$ for $1\leq j\leq k$. Then we have $C_{j}=g_{j}C_{0}$ and $gC_{0}=C_{k}=g_{k}C_{0}$. Since the face $\bar{F}_{i_{k}}$ of $C_{k-1}=g_{k-1}C_{0}$ and the face $F_{i_{k}}$ of $C_{0}$ are both of type $i_{k}$, we have $\bar{F}_{i_{k}}=g_{k-1}F_{i_{k}}$. Since $G_{\bar{F}_{i_{k}}}=G_{g_{k-1}F_{i_{k}}}=g_{k-1}G_{F_{i_{k}}}g_{k-1}^{-1}$, we get that $g_{i_{k}} = g_{k-1}h_{i_{k}}g_{k-1}^{-1}$, and hence $gC_{0} = g_{k-1}h_{i_{k}}g_{k-1}^{-1}(g_{k-1}C_{0})= (g_{i_{k-1}}\cdots g_{i_{1}})h_{i_{k}}C_{0}$, where $h_{i_{k}}\in H_{i_{k}}^{o}$ and $H_{i_{k}}$ is the isotropy group of the face $F_{i_{k}}$ of $C_{0}$. Proceeding in this way we get that $gC_{0} = h_{i_{1}} h_{i_{2}}\cdots h_{i_{k}}C_{0}$. Thus $(h_{i_{1}} h_{i_{2}}\cdots h_{i_{k}})^{-1}g\in G_{p}$. Hence any $g\in H$ can be written as $g=ka$ for some $k\in K$ and $a\in G_{p}$. Since $G_{p}$ acts trivially on $C_{0}$, we get that $G^{o}_{p}\subset K$. Thus $H^{o}\subset K$. So we get that $H^{o}= K$.

Since $G$ is compact and $G_{p}$ is closed in $G$, we see that $G_{p}$ and $N(G_{p})/G_{p}$ is compact. Thus by Proposition \ref{p1} and \ref{AT} $W=W_{G}(\Sigma)$ is finite, and hence $C_{0}$ has finitely many chamber faces. Therefore $H^{o}$ is generated by finitely many closed connected subgroup, and hence $H^{o}$ is closed in $G$ (see the answer by Mikhail Borovoi in mathoverflow \cite{Bo}). Thus $H^{o}$ is compact. Since $G_{p}$ is compact, $G_{p}/G^{o}_{p}$ is a finite group. Since $G^{o}_{p}\subset H^{o}$ and any $g\in H$ can be written as $g=ka$ for some $k\in K$ and $a\in G_{p}$, we get that $H/H^{o}$ is also a finite group, and hence $H$ is compact.

\noindent(2) Let $\Sigma$ be the section of $(M, G)$ through $p$. For any $x\in L_{p}^{\#}$, assume that $\tilde{x}\in G(x)\cap \Sigma$. By Lemma 3.1 $\tilde{x}\in L_{p}^{\#}\cap G(x)=H(x)$. Thus every $H$-orbit in $L_{p}^{\#}$ intersects $\Sigma$. Now it is easy to see that $\Sigma$ is a section for the isometric action $(L_{p}^{\#}, H)$. Hence the action of $H$ and $H^{o}$ on $L_{p}^{\#}$ is polar. By the proof of (1) we see that any $g\in H$ can be written as $g=ka$ for some $k\in K=H^{o}$ and $a\in G_{p}$. Since $G_{p}$ acts trivially on $\Sigma$, $g\tilde{x}=k\tilde{x}$. Thus $H(x)=H(\tilde{x})=H^{o}(\tilde{x})=H^{o}(x)$ for any $x\in L_{p}^{\#}$. By definition $L_{p}^{\#}$ is $G$-horizontal connected, hence $L_{p}^{\#}$ is also $H$-horizontal connected since any $G$-horizontal curve must be a $H$-horizontal curve.

\noindent(3) Let $\{x_{n}\}\subset L_{p}^{\#}$ be a sequence of points that converges to $x\in M$, i.e., $\lim_{n\rightarrow\infty} d(x_{n},x)=0$, where $d$ is the Riemannian distance on $M$. To prove that $L_{p}^{\#}$ is closed in $M$, it suffices to prove that $x\in L_{p}^{\#}$. $G(x)$ is a closed submanifold of $M$ since $G$ is compact. So for any $x_{i}\in \{x_{n}\}$, there is a $\bar{x}_{i}\in G(x)$ such that $d(x_{i},\bar{x}_{i})=d(x_{i},G(x))$. Thus the minimal geodesic $\gamma$ from $x_{i}$ to $\bar{x}_{i}$ is perpendicular to $G(x)$, it means $\gamma$ is horizontal. Then $\bar{x}_{i}\in L_{p}^{\#}$ since $x_{i}\in L_{p}^{\#}$. Since $d(\bar{x}_{i},x)\leq d(\bar{x}_{i},x_{i})+d(x_{i},x)\leq 2d(x_{i},x)$, we get that $\{\bar{x}_{n}\}$ converges to $x$. By (2) and Lemma 3.1, $\{\bar{x}_{n}\}\subset L_{p}^{\#}\cap G(x)=L_{p}^{\#}\cap G(\bar{x}_{1})=H(\bar{x}_{1})=H^{o}(\bar{x}_{1})$. By (1) $H^{o}$ is compact, thus $H^{o}(\bar{x}_{1})$ is a closed submanifold of $M$. Hence $x\in H^{o}(\bar{x}_{1})\subset L_{p}^{\#}$. It follows that $L_{p}^{\#}$ is closed in $M$.

\noindent(4) Let $\Sigma$ be the section of $(M, G)$ through $p$. By (2) $\Sigma$ is also a section for the polar action $(L_{p}^{\#}, H)$.
Set $N_{H}(\Sigma):=\{g\in H|\ g\Sigma=\Sigma\}$ and $Z_{H}(\Sigma):=\{g\in H|\ gx=x, \forall x\in\Sigma\}$. Then $W_{H}(\Sigma):= N_{H}(\Sigma)/Z_{H}(\Sigma)$ is the polar group of $(L_{p}^{\#}, H)$ associated to $\Sigma$. For any $g\in N_{G}(\Sigma)$, $gp\in \Sigma\subset L_{p}^{\#}$, thus by Lemma 3.1 $g\in S=H$. Thus $N_{G}(\Sigma)\subset H$, and hence $N_{H}(\Sigma)=N_{G}(\Sigma)\cap H=N_{G}(\Sigma)$. Also, $Z_{H}(\Sigma)=G_{p}=Z_{G}(\Sigma)$. Therefore $W_{H}(\Sigma)=W_{G}(\Sigma)$. Hence $M/G=\Sigma/W_{G}(\Sigma)=C_{0}=\Sigma/W_{H}(\Sigma)=L_{p}^{\#}/H$.
\end{proof}

\noindent {\it Proof of Theorem \ref{dl}.} For any $p\in M$, since $L_{p}^{\#}\cap M_{o}\neq\emptyset$, we can assume that $p\in M_{o}$ is a regular point. For any $x\in M$, set $A_{x}:=T^{\perp}_{x}L_{x}^{\#}$. Let $\mathfrak{g}=Lie(G)$ be the Lie algebra of $G$. Set $\mathfrak{m}:=\{X\in \mathfrak{g}\ |\ X(p)\in A_{p}\}$. Let $\gamma(t)$ be a horizontal geodesic with $\gamma(0)=p$ and $\gamma'(0)=\xi$. We extend $\xi$ to a horizontal vector field $\hat{\xi}$ on $G(p)$ by defining $\hat{\xi}_{gp}=g_{\ast}\xi$ for $g\in G$. Since $\hat{\xi}$ is $G$-equivariant, $\hat{\xi}$ is a basic normal field on $G(p)$. Let $X\in \mathfrak{m}$, and $\phi_{s}$ be the one-parameter subgroup of $G$ corresponding to X. Since the Killing field $X(t):=X(\gamma(t))$ is the variational field of the variation $\gamma_{s}(t):=\phi_{s}(\gamma(t))=\exp_{\phi_{s}p}(t\hat{\xi}_{\phi_{s}p})$, we get that $X(t)$ is a holonomy field along $\gamma(t)$. By Corollary 18.7 in \cite{LW}, $X(t)$ is parallel along $\gamma(t)$ and $X(t)\in A_{\gamma(t)}$ for all $t\in \mathbb{R}$. It follows that $X(x)\in A_{x}=T^{\perp}_{x}L_{p}^{\#}$ for all $x\in \Sigma_{p}$, where $\Sigma_{p}$ is the section of $(M, G)$ through $p$. For any $x\in M$, set $\mathfrak{m}(x):= span\{X(x)\ |\ X\in \mathfrak{m}\}$. Then we get that $\mathfrak{m}(x)=A_{x}$ for all $x\in \Sigma_{p}$.

Let $C_{0}\subset\Sigma_{p}$ be the unique chamber through $p$, and $H_{i}$ the isotropy group of the face $F_{i}$ of $C_{0}$. By Proposition \ref{p3}, for any point $q$ in the interior of $F_{i}$, we have $G_{q}=H_{i}$. For any $X\in \mathfrak{m}$ and $g\in G_{q}=H_{i}$, $(Ad_{g}X)(q)=(Ad_{g}X)(gq)=g_{\ast}(X(q))\perp g_{\ast}(T_{q}L_{p}^{\#})=T_{q}L_{p}^{\#}$ since $H_{i}\subset H$ preserves $L_{p}^{\#}$. Since $\mathfrak{m}(q)=A_{q}=T^{\perp}_{q}L_{p}^{\#}$, $Ad_{g}X\subset \mathfrak{m}$. Since $\dim(Ad_{g}(\mathfrak{m}))=\dim(\mathfrak{m})$, $Ad_{g}(\mathfrak{m})=\mathfrak{m}$ for any $g\in H_{i}$. Now by Theorem 3.2 (1) $Ad_{g}(\mathfrak{m})=\mathfrak{m}$ for any $g\in H^{o}$, where $H=\{g\in G\ |\ gL_{p}^{\#}=L_{p}^{\#}\}$. By Theorem 3.2 (2) and Lemma 3.1, $L_{p}^{\#}\cap G(p)=H(p)=H^{o}(p)$. Now for any $g\in H^{o}$ and $X\in \mathfrak{m}$, we have $(Ad_{g}X)(gp)=g_{\ast}(X(p))\perp g_{\ast}(T_{p}L_{p}^{\#})=T_{gp}L_{p}^{\#}$. Since $Ad_{g}(\mathfrak{m})=\mathfrak{m}$, $\mathfrak{m}(gp)=T^{\perp}_{gp}L_{p}^{\#}$. By the same reason as above, we get that $\mathfrak{m}(y)=T^{\perp}_{y}L_{p}^{\#}$ for any $y\in \Sigma_{gp}$ and $g\in H^{o}$, where $\Sigma_{gp}$ is the section through $gp$. Thus $\mathfrak{m}(x)=T^{\perp}_{x}L_{p}^{\#}$ for any $x\in L_{p}^{\#}$.

For any $x\in L_{p}^{\#}$ and $u\in T^{\perp}_{x}L_{p}^{\#}$, there is a Killing field $X\in \mathfrak{m}$ with $X(x)=u$, since $\mathfrak{m}(x)=T^{\perp}_{x}L_{p}^{\#}$. Also, $X(y)\in T^{\perp}_{y}L_{p}^{\#}$ for any $y\in L_{p}^{\#}$. Let $S$ be the Weingarten operator of $L_{p}^{\#}$. Then $\langle S_{u}v, v\rangle = -\langle\nabla _{v}X, v\rangle = 0$ for all $v\in T_{x}L_{p}^{\#}$, since $X$ is a Killing field. Therefore $S$ vanishes along $L_{p}^{\#}$. It follows that $L_{p}^{\#}$ is totally geodesic in $M$. Since $M$ has nonnegative sectional curvature, $L_{p}^{\#}$ is also nonnegatively curved.

The rest of the proof follows from Theorem \ref{generate}.
$\hfill \square$\\

\noindent {\it Proof of Theorem \ref{df}.} Suppose the dual foliation $\mathcal{F}^{\#}$ has more than one leaf. By Theorem \ref{dl} every leaf of $\mathcal{F}^{\#}$ is closed. Now by Theorem 2 and Theorem 3 in \cite{wi} $\mathcal{F}^{\#}$ is a singular Riemannian foliation with closed leaves, thus the leaves of $\mathcal{F}^{\#}$ are the fibers of a submetry $\sigma: M\rightarrow Y$, where $Y=M/\mathcal{F}^{\#}$ is the leaf space of $\mathcal{F}^{\#}$. Fix a $p\in M$, and set $H=\{g\in G\ |\ gL_{p}^{\#}=L_{p}^{\#}\}$. It is easy to see that any leaf of $\mathcal{F}^{\#}$ can be written as $L_{gp}^{\#}$ for some $g\in G$. By Lemma 3.1 $L_{gp}^{\#}=gL_{p}^{\#}$ for any $g\in G$, thus the action of $G$ on $M$ induces a transitive isometric action of $G$ on the space $Y$ of dual leaves. Since $H$ is the stabilizer group of $L_{p}^{\#}$, $Y$ is the homogeneous space $G/H$. By Proposition 9.1 in \cite{wi} $\sigma$ is a $G$-equivariant Riemannian submersion. The rest of the proof follows from Theorem \ref{dl} since any fiber of $\sigma$ can be written as $L_{gp}^{\#}=gL_{p}^{\#}$ for some $g\in G$.
$\hfill \square$\\

 \end{document}